\documentclass[11pt]{amsart}
\usepackage{amssymb,amsmath,amsthm,verbatim,enumitem}
\usepackage{fullpage}
\usepackage{xcolor}
\usepackage{url}
\usepackage{mathrsfs}
\usepackage[all]{xy}
  \SelectTips{cm}{10}
  \everyxy={<2.5em,0em>:}
\usepackage{tikz}
\usepackage{picinpar} 

\usepackage{fancyhdr}
\usepackage{ifpdf}
  \ifpdf
    \usepackage{hyperref} 
  \else
    
    \newcommand{\href}[2]{#2}
  \fi

\theoremstyle{plain}
  \newtheorem{lemma}[equation]{Lemma}
  \newtheorem{proposition}[equation]{Proposition}
  \newtheorem{theorem}[equation]{Theorem}
  \newtheorem{corollary}[equation]{Corollary} 
   
    \newtheorem{question}[equation]{Question}
    \newtheorem{conjecture}[equation]{Conjecture}

\theoremstyle{definition}
  \newtheorem{definition}[equation]{Definition}

\theoremstyle{remark}
  \newtheorem{remark}[equation]{Remark}

\renewcommand{\thesection}{\arabic{section}}
\renewcommand{\theequation}{\thesection.\arabic{equation}}

 \DeclareFontFamily{U}{manual}{}
 \DeclareFontShape{U}{manual}{m}{n}{ <->  manfnt }{}
 \newcommand{\manfntsymbol}[1]{%
    {\fontencoding{U}\fontfamily{manual}\selectfont\symbol{#1}}}

\makeatletter
   \@addtoreset{section}{part}
   \@addtoreset{equation}{section}
   \@addtoreset{footnote}{section}

    {\hspace*{\fill}$\lrcorner$\endgraf\endgroup\end{trivlist}}
 \newenvironment{example}[1][]{
   \refstepcounter{equation}
   \begin{proof}[Example~\theequation%
   \@ifnotempty{#1}{ (#1)}.]
   }
  {\end{proof}}
\makeatother

  \DeclareFontFamily{OT1}{pzc}{}
  \DeclareFontShape{OT1}{pzc}{m}{it}{<-> s * [1.100] pzcmi7t}{}
  \DeclareMathAlphabet{\mathpzc}{OT1}{pzc}{m}{it}

\newif\ifhascomments \hascommentstrue
\ifhascomments
  \newcommand{\dan}[1]{{\color{red}[[\ensuremath{\bigstar\bigstar\bigstar} #1]]}}
  \newcommand{\matt}[1]{{\color{red}[[\ensuremath{\spadesuit\spadesuit\spadesuit} #1]]}}
\else
  \newcommand{\dan}[1]{}
  \newcommand{\matt}[1]{}
\fi

\renewcommand{\AA}{\mathbb{A}}

\DeclareMathOperator{\Aut}{\ensuremath{\mathcal{A}\kern-.125em\mathpzc{ut}}}

\newcommand{\Pro}{\mathbb P}

\newcommand{\D}{\mathcal D}

\renewcommand{\emptyset}{\varnothing}

\newcommand{\E}{\mathcal E}
\DeclareMathOperator{\Endo}{\ensuremath{\mathcal{E}\kern-.125em\mathpzc{nd}}}

\newcommand{\Proj}{\mathbb P}

\newcommand{\GG}{\mathbb G}
\newcommand{\ix}{\mathcal X}

\DeclareMathOperator{\Projj}{Proj}

\let\hom\relax

\DeclareMathOperator{\hom}{Hom}
\DeclareMathOperator{\hgt}{ht}

\DeclareMathOperator{\Hom}{\ensuremath{\mathcal{H}\kern-.125em\mathpzc{om}}}

\renewcommand{\L}{\mathcal L}

\renewcommand{\O}{\mathcal O}

\newcommand{\QQ}{\mathbb Q}

\DeclareMathOperator{\rk}{rk}

\renewcommand{\setminus}{\smallsetminus}

\DeclareMathOperator{\spec}{Spec}

\newcommand{\U}{\mathcal U}
\newcommand{\iC}{\mathcal C}
\newcommand{\iI}{\mathcal I}

\newcommand{\W}{\mathcal W}
\newcommand{\X}{\mathcal{X}}
\newcommand{\x}{\text{x}}
\newcommand{\Y}{\mathcal{Y}}
\newcommand{\Z}{\mathcal{Z}}
\newcommand{\ZZ}{\mathbb{Z}}

 \def\ari[#1]{\ar@{^(->}[#1]}
 \def\are[#1]{\ar[#1]^{\txt{\'et}}}
 \def\areh[#1]{\ar[#1]|{\txt{$H$-eq}}^{\txt{\'et}}}
 \def\ars[#1]{\ar@{->>}[#1]}
 \newcommand{\dplus}{\ar@{}[d]|{\mbox{$\oplus$}}}
 \newcommand{\dtimes}{\ar@{}[d]|{\mbox{$\times$}}}

\DeclareMathOperator{\codim}{codim}

\DeclareMathOperator{\op}{op}
\DeclareMathOperator{\st}{st}
\DeclareMathOperator{\pst}{pst}
\DeclareMathOperator{\tst}{tst}
\DeclareMathOperator{\tame}{tame}

\DeclareMathOperator{\Pic}{Pic}

\title{Towards an Intersection Chow Cohomology Theory for GIT Quotients}

\author{Dan Edidin}
\address{Department of Mathematics, University of Missouri, Columbia MO 65211}
\email{edidind@missouri.edu}

\author{Matthew Satriano}
\address{University of Waterloo \\
Department of Pure Mathematics \\
Waterloo, Ontario \\
Canada  N2L 3G1}

\thanks{The first author was supported by Simons Collaboration Grant
  315460. The second author was partially supported by a Discovery Grant from the National Science and Engineering Board of Canada.}

\email{msatriano@uwaterloo.ca}

\subjclass[2010]{14C15, 14F43}

\begin{document}

\begin{abstract}
  We study the Fulton-MacPherson operational Chow rings of good moduli spaces
  of properly stable, smooth, Artin stacks. Such spaces are \'etale locally
  isomorphic to geometric invariant theory quotients of affine schemes, and are therefore natural extensions of GIT quotients. Our main result is that, with
  $\QQ$-coefficients, every operational class can be represented by a {\em topologically strong} cycle on the corresponding stack. Moreover, this cycle is unique modulo rational equivalence on the stack.
  Our methods also allow us to prove that if $X$ is the good moduli space of a properly stable, smooth, Artin stack then the natural map $\Pic(X)_\QQ \to A^1_{\op}(X)_\QQ$, $L \mapsto c_1(L)$ is an isomorphism.
\end{abstract}

\maketitle
\tableofcontents


\section{Introduction}
\label{sec:intro}
A long standing problem is to extend the classical intersection
product to singular varieties. Motivated by topology, one hopes to
construct an ``intersection Chow cohomology'' theory analogous to
Goresky and MacPherson's intersection homology. There have been
various approaches from this point of view using motivic theories, for
example \cite{Corti-Hanamura,Corti-HanamuraII, Friedlander-Ross}.

Earlier, Fulton and MacPherson \cite{FulMac:81} defined a formal Chow cohomology ring,
which we call the operational Chow ring, and proved that it equals the
classical intersection ring for smooth varieties. While the
operational Chow ring enjoys many natural functorial properties, the
product structure, which is given by composition of operations, does
not have a natural interpretation in terms of intersecting
subvarieties.

In this paper we focus on a class of varieties which we call {\em
  reductive quotient varieties}. These are varieties (or more
generally algebraic spaces) which are
\'etale locally
good
quotients of smooth varieties by reductive groups. When a reductive
group acts properly the quotient variety has finite quotient
singularities, and there is a well developed geometric intersection
theory on such varieties \cite{Vis:89, EdGr:98}. However, when
the stabilizers are positive dimensional the singularities are worse
and there is no intersection ring. For example, the cone over the
quadric surface is the good quotient of $\AA^4$ by $\GG_m$. More
generally, the Cox construction \cite{Cox:95} shows that all normal toric
varieties are good quotients of open sets in affine space. However,
only simplicial toric varieties have finite quotient singularities.

If $X = Z/G$ is a quotient variety with $Z$ smooth, then the quotient
map $Z \stackrel{p} \to X$ induces a stratification of $X$ coming from
the stratification of $Z$ by the dimension of stabilizers. Our main
result is that every operational class $c \in A^*_{\op}(X)_\QQ$ can be
represented by a cycle which satisfies certain transversality
conditions with respect to this stratification. Moreover, this cycle
is unique modulo rational equivalence on the associated quotient stack $[Z/G]$.
One consequence of this uniqueness is that the natural map
$\Pic(Z/G) \to A^1_{\op}(Z/G)$, $L \mapsto c_1(L)$ is surjective after tensoring with $\QQ$. In other words an integer multiple of every codimension-one operational class can be identified with the first Chern class of a line bundle.
(See
Theorem \ref{thm:pic-surjectivity} and the remarks thereafter.)



\subsection{Statement of results}
The point of view of this paper is stack-theoretic. A reductive
quotient variety $X = Z/G$ is the good moduli space \cite{Alp:13} of a smooth stack $\X =
[Z/G]$. Given a stack $\X$ with good moduli space $X$, we say that an
integral substack $\Y \subseteq \X$ is {\em strong} if $\Y$ is saturated
with respect to the good moduli space morphism $\pi \colon \X \to X$, i.e.~$\Y = \pi^{-1}(\pi(\Y))$. We say $\Y$ is {\em topologically strong} if $\Y = \pi^{-1}(\pi(\Y))_{red}$. We prove in Proposition \ref{prop.stratify} that $X$ has a stratification by subspaces with finite quotient singularities, and that if $\Y$ is topologically strong, then the image $Y = \pi(\Y)$ is a closed subspace of $X$ satisfying a transversality condition with respect to this stratification.

We
define the {\em relative strong Chow group} $A^k_{\st}(\X/X)$ to be
the subgroup of the Chow group $A^k(\X)$ generated by classes of
strong integral substacks of codimension $k$ and the
{\em relative topologically strong Chow group} $A^k_{\tst}(\X/X)$
to be the subgroup generated by topologically strong cycles.
Note that while the ordinary Chow groups of an Artin stack can be non-torsion in
arbitrarily high degree, the relative strong and topologically strong Chow groups vanish in
degree greater than the dimension of the good moduli space $X$, and so these Chow groups reflect more of the geometry of $X$. 
The strong relative Chow group of an Artin stack was originally defined in \cite{EdSa:16}. We show in Example \ref{ex:tst-st-differ} that $A^*_{\st}(\X/X)$ need not equal $A^*_{\tst}(\X/X)$ even rationally.

We say that an integral stack $\X$ with good moduli space $X$ is {\em
  properly stable} if there is a non-empty open set $\U \subseteq \X$
which is a tame stack in the sense of \cite{AOV:08} and which is saturated with respect to the good
moduli space morphism $\X \to X$.

\begin{theorem}
\label{thm:inj-pullback}
Let $\X$ be a smooth connected properly stable Artin stack with good moduli space $\pi\colon\X\to X$. Consider the map $\pi^*\colon A^k_{\op}(X)_\QQ \to A^k(\X)_\QQ$ defined by $c \mapsto c \cap [\X]$. Then $\pi^*$ is injective and its image is contained in $A^k_{\tst}(\X/X)_\QQ$. In particular every operational class $c\in A^*_{\op}(\X)_\QQ$ determines a topologically strong cycle which is unique modulo rational equivalence in $A^*(\X)_\QQ$.
\end{theorem}

As a consequence of Theorem \ref{thm:inj-pullback} we are able to prove the following 
theorem about Picard groups of properly stable good moduli spaces of smooth Artin stacks.
\begin{theorem} \label{thm:pic-surjectivity}
Let $\pi\colon\X\to X$ be a properly stable good moduli space morphism with $\X$ smooth.
Then the natural map $\Pic(X)_\QQ \to A^1_{\op}(X)_\QQ$ is 
an isomorphism.
\end{theorem} 

\begin{corollary}
Let $Z$ be a smooth projective variety with a linearized action of a reductive group $G$ and let $X = Z/\!/G$ be the GIT quotient. If $Z^{ps} \neq \emptyset$ then the natural map $\Pic(X)_\QQ \to A^1_{\op}(X)_\QQ$ is an isomorphism 
\end{corollary}
(Here $Z^{ps}$ refers to the set of properly stable point for the linearized action of $G$ \cite[Definition 1.8]{MFK:94}. See also Remark \ref{remark:ps}.)
\begin{remark} \label{rem.injective}
For any normal algebraic space the map $\Pic(X) \to A^1_{\op}(X)$
is injective with integer coefficients because the map from $\Pic(X)$ to the divisor class
group $A^1(X)$ is injective \cite[21.6.10]{EGA4} and this map
factors through the natural map $\Pic(X) \to A^1_{\op}(X)$. Thus to establish
Theorem \ref{thm:pic-surjectivity} we need only prove surjectivity because
the good moduli space of any smooth Artin stack is normal \cite[Theorem 4.16]{Alp:13}
\end{remark}

\begin{remark} In characteristic 0, GIT quotients have rational singularities by Boutot's theorem \cite{Bou:87}. It follows from 
\cite[Proposition 12.4]{KoMo:92}
that if $k = {\mathbb C}$ then the map $\Pic(X)_\QQ \to A^1_{\op}(X)_\QQ$ is an isomorphism (cf. \cite[pp.4--5]{FMSS:95}).
\end{remark}

\begin{remark}
If $k = {\mathbb C}$ and $X$ is spherical (for example if $X$ is a toric variety) then map
$\Pic(X) \to A^1_{\op}(X)$ is known to be an isomorphism with integer coefficients. See Totaro's paper \cite[p22.]{Totaro} where he attributes this fact 
to Brion.
\end{remark}

From Theorem \ref{thm:inj-pullback}, we know that the image of $\pi^*$ is contained in $A^*_{\tst}(\X/X)_\QQ$. On the other hand, Example \ref{ex:tst-st-differ} shows that the image is not generally equal to $A^*_{\tst}(\X/X)_\QQ$. In Conjecture \ref{conj:image-pi-pst}, we state a precise conjectural description of the image. 
The following theorem 
summarizes some partial results in this direction; see also Theorem \ref{thm:conj-pst}.

\begin{theorem} \label{thm:surj-pullback}
Let $\X$ be a properly stable smooth Artin stack with good moduli space $\pi\colon\X\to X$. Then the following hold:
  \begin{enumerate}[label=\emph{(\alph*)}, ref=\alph*]
    \item \label{surj::quot-sing}
If $X$ is smooth or has finite quotient singularities, then $A^*(X)_\QQ\simeq A^*_{\op}(X)_\QQ\simeq A^*_{\tst}(\X/X)_\QQ$, where the second isomorphism is induced by $\pi^*$.

  \item \label{surj::A1} If $X$ has an ample line bundle, then $\pi^* A^1_{\op}(X) = A^1_{\st}(\X/X)$. In particular, $\pi^*$ induces an isomorphism $A^1_{\op}(X)_\QQ \simeq A^1_{\st}(\X/X)_\QQ$.

\item \label{surj::reg-embedded}
If $\Z \subseteq \X$ is a regularly embedded strong substack of arbitrary
codimension $k$ then $[\Z] \in \pi^*(A^k_{\op}(X))_\QQ$.

\item \label{surj::dimX}
Assume $k$ equals $\dim X -1$ or $\dim X$. 
Then the map $\pi^* \colon A^k_{\op}(X)_\QQ \to A^k_{\tst}(\X/X)_\QQ$ is an 
isomorphism. If in addition the maximal saturated tame substack of $\X$ 
is representable, then $A^k_{\st}(\X/X)_\QQ=A^k_{\tst}(\X/X)_\QQ$. 
\end{enumerate}
\end{theorem}

\begin{corollary} \label{cor:lowdim}
If $\dim X \leq 3$, the maximal saturated tame substack 
of $\X$ is representable, and $X$ has an ample line bundle, then the 
map $A^*_{\op}(X)_\QQ \to A^*_{\st}(\X/X)_\QQ$ is an isomorphism.
\end{corollary}

\begin{remark}
\label{indep-of-stack}
 Theorem \ref{thm:surj-pullback}(\ref{surj::quot-sing}) implies that if the good moduli space
  $X$ is smooth or has finite quotient singularities, then $A^*_{\tst}(\X/X)_\QQ\subseteq A^*(\X)_\QQ$ is a subring, i.e.~it is closed under intersection products. Furthermore, it shows that $A^*_{\tst}(\X/X)_\QQ$ is independent of the stack $\X$. In other words,
  any two stacks $\X'$ and $\X$ with good moduli space $X$
  have rationally isomorphic topologically strong Chow groups. For
  Deligne-Mumford or tame stacks with coarse moduli space $X$ this is well
  known. Our result shows that even if the stack $\X$ is not
  tame, there is still a canonical subring $A^*_{\tst}(\X/X)_\QQ\subseteq A^*(\X)_\QQ$ which captures the rational
  Chow ring of the moduli space $X$. 
\end{remark}

\section{Background results}
\label{sec:background}
Throughout this paper, all schemes and stacks are assumed to be of finite type over
an algebraically closed
field. All algebraic stacks are assumed to have affine diagonal. An Artin stack is {\em tame} if it has finite inertia and
satisfies any (and hence all) of the equivalent conditions of \cite[Theorem 3.2]{AOV:08}. In characteristic 0, all Deligne-Mumford stacks with finite inertia are tame.

\subsection{Intersection theory on schemes and stacks}

Let $X$ be a scheme or algebraic space. We denote
by $A_n(X)$ the Chow group of $n$-dimensional cycles modulo rational equivalence and we denote by $A_*(X)$ the direct sum of all Chow groups.
Unless otherwise stated, we assume that a scheme is equidimensional and use the notation $A^k(X)$ to denote the Chow group of codimension-$k$ cycles modulo rational equivalence. We denote by $A^k_{\op}(X)$
the codimension-$k$ operational Chow group of $X$ as defined in \cite[Chapter 17]{Fulton}. By definition, an element $c \in A^k_{\op}(X)$ is an assignment,
for every morphism $T \to X$ a homomorphism $c_T \colon A_{*}(T) \to A_{*-k}(T)$
  which is compatible with the basic operations of Chow groups (flat and lci pullbacks, as well as proper pushforward). If $\alpha \in A_*(T)$ we write
  $c \cap \alpha$ for $c_T(\alpha)$. The group $\bigoplus A^k_{\op}(X)$ is a graded ring
  with multiplication given by composition. When $X$ is smooth, the Poincar\'e
  duality map
  $A^*_{\op}(X) \to A^*(X)$, $c \mapsto c \cap [X]$ is an isomorphism of rings
  \cite[Cor 17.4]{Fulton}
  where the product on $A^*(X)$ is the intersection product defined in \cite[Chapter 8]{Fulton}.

  There is also an intersection theory for stacks which was developed in \cite{EdGr:98} and \cite{Kre:99} building on earlier work of Gillet \cite{Gil:84} and Vistoli \cite{Vis:89}.  If $\X = [Z/G]$ is a quotient stack then we can identify
  the Chow group $A^k(\X)$ with the equivariant Chow group $A^k_G(X)$ defined in
  \cite{EdGr:98}.
  For any stack we can define an operational Chow ring in a manner similar to the definition for schemes. Precisely, an element $c \in A^k_{\op}(\X)$
  is an assignment to every morphism from a scheme  $T \to \X$ an operation
  $c_T \colon A^*(T) \to A^{*-k}(T)$ with the usual compatibilities.
  When $\X = [Z/G]$ is a quotient we can identify the operational Chow
  ring $A^*_{\op}(\X)$ with
  the operational equivariant Chow ring of $A^*_{\op, G}(Z)$ defined in \cite{EdGr:98}. Again if $\X$ is smooth, there is a Poincar\'e duality isomorphism
  $A^*_{\op}(\X) \to A^*(\X)$, $c \mapsto c \cap [\X]$ \cite{EdidinHandbook}.

  Let $\ix$ be an integral Artin stack which can be stratified by quotient
  stacks\footnote{Any stack with a good moduli space necessarily has affine
    stabilizers and so can be stratified by quotient stacks by \cite[Proposition 3.5.9]{Kre:99}} and let  $\pi \colon \ix \to X$ be
  any morphism to an algebraic space. The following proposition gives the
  construction of an evaluation map 
$A^*_{\op}(X) \to A^*(\ix)$, $c \mapsto c \cap [\ix]$.
\begin{proposition}
  There exists an evaluation map $A^*_{\op}(X) \to A^*(\ix)_\QQ$, $c \mapsto
  c \cap [\ix]_\QQ$. Moreover, if $\ix$ is a quotient stack or if the characteristic of the ground field is $0$
  then this map can defined with $\ZZ$-coefficients.
\end{proposition}
\begin{proof}
  For general $\ix \to X$, we can apply Chow's lemma and deJong's alteration theorem to construct a generically finite morphism of degree $d$,
  $g \colon X' \to X$ where $X'$ is smooth and has an ample line bundle.
 
  Since $X'$ is smooth and has an ample line bundle, we can, by the Riemann-Roch theorem, express the image of $g^*c$ in $A^*(\X')_\QQ$  as $p(E_1, \ldots , E_r)$ where $p$ is a polynomial in the Chern classes of vector bundles $E_1, \ldots E_r$ on $X'$.

  Let $\ix' = X' \times_X \ix$ be the stack obtained by base change, so we have a catersian diagram.
 \[
  \xymatrix{
  \X'\ar[r]^f\ar[d]^-{\pi'} & \X\ar[d]^-{\pi}\\
  X'\ar[r]^g & X
  }
  \]
  By \cite[Theorem 2.1.12(vii)]{Kre:99} Chern class operations
  are defined on the Chow groups of any Artin stack. Thus we define $c \cap [\ix]$ by the formula
  $${1\over{d}}f_* \left(p((\pi')^*E_1, \ldots , (\pi')^*E_r)) \cap [\ix']\right).$$
  Standard arguments using the projection formula show that this is independent of choice alteration $X'$.
  
  If we work over a field of characteristic 0 then using resolution of
  singularities we can assume that $X' \to X$ is birational. Let
  $h \colon \tilde{\ix'} \to \ix'$ be a resolution of
  singularities\footnote{The existence of resolution of singularities
    for Artin stacks follows from functorial resolution of
    singularities for schemes. Since resolutions are functorial for
    smooth morphisms they can be constructed for Artin stacks which
    are locally schemes in the smooth topology. The functoriality also
    implies that the resolution of singularities morphism is
    representable.} of the  stack $\ix'$ obtained by base change along the
  morphism $g \colon X' \to X$. Let $p = \pi' \circ h$ and
  $\tilde{f} = f \circ h$. Then we have a commutative diagram:

   \[
  \xymatrix{
  \tilde{\X'}\ar[r]^{\tilde{f}}\ar[d]^-{p} & \X\ar[d]^-{\pi}\\
  X'\ar[r]^g & X
  }
  \]
  Both $\tilde{\ix'}$ and $X'$ are smooth so the morphism $p$ is l.c.i.
  Also the stack
  $\tilde{\ix'}$ can be stratified by quotient stacks
  since it admits a representable morphism to the stack $\ix$ which has the same property.

  By \cite[Theorem 2.1.12(xi)]{Kre:99} there is an
  l.c.i. pullback $p^* \colon A^*(X') = A^*_{\op}(X') \to
  A^*(\tilde{\ix'})$.  We define $$c \cap [\ix] =
  \tilde{f}_*\left(p^*g^*c \cap [\tilde{\ix'}]\right).$$ Again this
  formula is independent of choices of resolutions.
  
  Finally if $c \in A^k_{\op}(X)$ and $\ix = [Z/G]$ is a quotient stack, then 
  $A^k(\ix)$ is identified with
  $A^k((Z \times U)/G)$ for some open set $U$ in a representation of $G$.
  In this case we identify $c \cap [\ix]$ with
  $c \cap [(Z \times U)/G] \in A^k((Z \times U)/G)$.
\end{proof}

  \subsection{Properly stable good moduli spaces and Reichstein transforms}
  We briefly review the material on properly stable good moduli spaces and Reichstein transforms from \cite{EdRy:17, EdSa:16}. 

  \begin{definition}[{\cite[Definition 4.1]{Alp:13}}]
Let $\ix$ be an algebraic stack and let $X$ be an algebraic space. We say
that $X$ is a {\em good moduli space of $\ix$} if there is a morphism
$\pi \colon \ix \to X$ such that
\begin{enumerate}
\item $\pi$ is {\em cohomologically affine} meaning that the pushforward functor $\pi_*$
on the category of quasi-coherent ${\mathcal O}_\ix$-modules is exact.

\item $\pi$ is {\em Stein} meaning that the natural map ${\mathcal O}_X \to \pi_* {\mathcal O}_\ix$ is an isomorphism.
\end{enumerate}
More generally, a morphism $\pi \colon \ix \to \Y$ of algebraic stacks is a {\em good moduli space morphism} if it satisfies conditions (1) and (2) above.
\end{definition}
\begin{remark} By \cite[Theorem 6.6]{Alp:13}, a good moduli space morphism
$\pi\colon\ix\to X$ is the universal morphism from $\ix$ to an algebraic space. That is, if $X'$ is an algebraic space then any morphism $\ix \to X'$ factors through a morphism $X \to X'$. Consequently $X$ is unique up to unique isomorphism, so we will refer to $X$ as \emph{the} good moduli space of $\ix$.
\end{remark}

\begin{remark}
If $\ix = [Z/G]$ where $G$ is a linearly reductive algebraic group
then the statement that $X$ is a good moduli space for $\ix$ is 
equivalent to the
statement that $X$ is the good quotient of $Z$ by $G$. 
\end{remark}

\begin{definition}[{\cite{EdRy:17}}] \label{def.stablegms}
Let $\ix$ be an Artin with good moduli space $X$ and
let $\pi \colon \ix \to X$ be the good moduli space morphism. We say that a closed point
of  $\ix$ is
{\em
   stable} if $\pi^{-1}(\pi(x)) = x$ under the induced map of
  topological spaces $|\X| \to |X|$. A point $x$ of $\X$ is {\em
    properly stable} if it is stable and the stabilizer of $x$ is finite.

We say $\ix$ is  stable (resp.~properly stable) if there is a good moduli
space $\pi \colon \ix\to X$ and the set of stable (resp.~properly stable) points is non-empty. Likewise we say that $\pi$ is a stable (resp.~properly stable) good moduli space morphism.
\end{definition}
\begin{remark} \label{remark:ps}
This definition is modeled on GIT. If $G$ is a linearly reductive group
and $X^{ss}$ is the set of semistable points for a linearization of the 
action of $G$ on a projective variety $X$ then a (properly) stable point
of $[X^{ss}/G]$ corresponds to a (properly) stable orbit in the sense of GIT. 
The stack $[X^{ss}/G]$ is stable if and only if $X^{s} \neq \emptyset$. Likewise
$[X^{ss}/G]$ is properly stable if and only if $X^{ps} \neq \emptyset$. As is the case for GIT quotients, the set of stable (resp. properly stable points)
is open \cite{EdRy:17}.

We denote by $\ix^s$ (resp.~$\ix^{ps}$) the open substack of $\ix$ consisting of
stable (resp.~properly stable) points. The stack $\ix^{ps}$ is the
maximal tame substack of $\ix$ which is saturated with
respect to the good moduli space morphism $\pi\colon\ix \to X$.  In
particular, a stack $\ix$ with good moduli space $X$ is properly
stable if and only if it contains a non-empty {\em saturated}
tame open substack.
\end{remark}

The following definition is a straightforward extension of
the one  originally made in \cite{EdMo:12}. 
\begin{definition} \label{def.reichstein}
  Let $\pi \colon \X \to X$ be a good moduli space morphism and let
  $\iC \subseteq \X$ be a closed substack. Let $f \colon\widetilde{\X}\to \X$
  be the blow-up along $\iC$. 
  The {\em Reichstein transform} with center $\iC$, is the stack $R(\X, \iC)$
  obtained by deleting the strict transform of the saturation
  $\pi^{-1}(\pi(\iC))$ in the blow-up of $\ix$ along $\iC$.
\end{definition}
\begin{theorem}\label{thm.reichstein}
  \cite[Proposition 3.4, Proposition 4.5]{EdRy:17}
  Let $\pi \colon \ix \to X$ be a good moduli space morphism with $\ix$ smooth
  and let
  $\iC \subseteq \X$ be a smooth closed substack with sheaf of ideals
  ${\mathcal I}$. Then $R(\X, \iC) \to \Proj(\bigoplus \pi_*({\mathcal I}^n))$ is
  a good moduli space morphism.
\end{theorem}

Let $\pi \colon \X \to X$ be a stable good moduli space morphism with
$\X$ smooth and irreducible. Let $\Y \subseteq \X$ be the locus of
points with maximal dimensional stabilizer $n$. By \cite{EdRy:17}, $\Y$
is a closed smooth substack of $\X$. Let $f \colon \X' \to \X$ be the
Reichstein transform along $\Y$ and let $g \colon X' \to X$ be the
induced morphism of good moduli spaces.  If $\Y$ is a proper closed
substack then by \cite[Proposition 5.3]{EdRy:17}, $g$ is a projective
birational morphism and the stabilizer of every point of $\X'$ is
strictly less than $n$. As a consequence, after a finite sequence of
Reichstein transforms we may obtain a stack $\X''$ where the dimension of the
stabilizer is constant. If $\X$ is properly stable then $\X''$
is a tame stack; otherwise $\X''$ is a gerbe over a tame stack.

\section{Strong and topologically strong Chow groups}
We begin with the two central definitions of the paper, the first of
which having previously been introduced in \cite{EdSa:16}. Note first
that if $\pi \colon \X \to X$ is a good moduli space morphism and
$\Z\subseteq\X$ is a closed substack, then $\pi(\Z)\subseteq X$
inherits a natural subscheme structure. Indeed, if $\iI\subseteq\O_\X$
is the coherent sheaf of ideals defining $\Z$, then since $\pi_*$ is
Stein and cohomologically affine, we see
$\pi_*\iI\subseteq\pi_*\O_\X=\O_X$ is a coherent sheaf of ideals and
set-theoretically cuts out $\pi(\Z)$, thereby giving $\pi(\Z)$ a
natural scheme structure.

\begin{definition}
\label{def:strong}
Let $\X$ be an irreducible Artin stack with stable good moduli space $\pi \colon \X \to X$. A
closed integral substack $\Z \subseteq \X$ is {\em strong} if $\codim_\X\Z
=\codim_X\pi(\Z)$ and $\Z$ is saturated with respect to $\pi$, i.e.~$\pi^{-1}(\pi(\Z)) = \Z$ as stacks. We say $\Z$ is {\em topologically strong} if
$\codim_\X\Z=\codim_X\pi(\Z)$ and $\pi^{-1}(\pi(\Z))_{red} = \Z$.
\end{definition}

\begin{remark} Note that if $\X$ is tame then $\X$ has a coarse moduli space
  $X$ 
  which is also the good moduli space \cite[Example 8.1]{Alp:13}.
  If $\X \to X$ is the coarse/good moduli space morphism, then any integral substack $\Z \subseteq \X$ is topologically strong. If $\X \to X$ is properly stable then $\dim \X = \dim X$, so
  $\dim \Z = \dim \pi(Z)$ if $\Z$ is strong or topologically strong.
  \end{remark}

\begin{lemma}
\label{l:strong-etale-loc}
Let $\pi\colon\X\to X$ be a stable good moduli space morphism and $\Z\subseteq\X$ an irreducible closed substack. Consider the cartesian diagram
\[
\xymatrix{
\X'\ar[r]^f\ar[d]^-{\pi'} & \X\ar[d]^\pi\\
X'\ar[r]^g & X
}
\]
with $g$ an \'etale cover. Let $\Z'=f^{-1}(\Z)$ and $Z'=\pi'(\Z')$. Then $\Z$ is strong (resp.~topologically strong) with respect to $\pi$ if and only if $\Z'$ is reduced, $\codim_{\X'}\Z'=\codim_{X'} Z'$, and $\Z'={\pi'}^{-1}(Z')$ (resp.~$\Z'=({\pi'}^{-1}(Z'))_{red}$).
\end{lemma}
\begin{proof}
Notice that $\Z$ integral if and only if $\Z$ is reduced, and since $f$ is an \'etale cover, this holds if and only if $\Z'$ reduced.

Next, we show that $Z'=g^{-1}(\pi(\Z))$. The diagram in the statement of the lemma is cartesian, and so
\[
\xymatrix{
\Z'\ar[r]\ar[d] & \Z\ar[d]\\
X'\ar[r]^-{g} & X
}
\]
is as well. This factors into a commutative diagram
\[
\xymatrix{
\Z'\ar[r]\ar[d] & \Z\ar[d]\\
g^{-1}(\pi(\Z))\ar[r]\ar[d] & \pi(\Z)\ar[d]\\
X'\ar[r]^-{g} & X
}
\]
and the bottom square is cartesian, so the top square is as well. Now $\Z\to\pi(\Z)$ is surjective, so $\Z'\to g^{-1}(\pi(\Z))$ is as well, showing $Z'=g^{-1}(\pi(\Z))$.

Since $f$ and $g$ are \'etale, $\dim\Z=\dim\Z'$ and $\dim\pi(\Z)=\dim g^{-1}\pi(\Z)=\dim Z'$. Therefore, $\codim_\X\Z=\codim_X\pi(\Z)$ if and only if
$\codim_{\X'}\Z'=\codim_{X'} Z'$.

Lastly, since $f$ is an \'etale cover, the canonical map $\Z\to\pi^{-1}\pi(\Z)$ is an isomorphism if and only if $\Z'\to f^{-1}\pi^{-1}\pi(\Z)=\pi'^{-1}Z'$ is an isomorphism. 
Similarly, we show $\Z=(\pi^{-1}\pi(\Z))_{red}$ if and only if $\Z'=({\pi'}^{-1}(Z'))_{red}$ under the presence of the equivalent hypotheses: $\Z$ is integral and $\Z'$ is reduced. To see this, note that the canonical map $\Z\to(\pi^{-1}\pi(\Z))_{red}$ is an isomorphism if and only if $\Z'\to f^{-1}((\pi^{-1}\pi(\Z))_{red})=(f^{-1}\pi^{-1}\pi(\Z))_{red}=(\pi'^{-1}Z')_{red}$ is an isomorphism.
\end{proof}

\begin{remark}[Local structure of strong and topologically strong substacks]
\label{rmk:local-structure-strong}
Let $\pi\colon\X\to X$ be a good moduli space morphism. By the local structure theorem of \cite{AHR:15}, for any point $x$ of $\X$, letting $G_x$ denote the stabilizer of $x$, there is an \'etale neighborhood of $x$ isomorphic to $[U/G_x]$ such that the diagram
$$\xymatrix{[U/G_x] \ar[r] \ar[d] & \X \ar[d]\\ U/G_x \ar[r] & X}$$
is cartesian and the horizontal maps are \'etale. Shrinking $U$ is necessary, we may assume it is affine.

By Lemma \ref{l:strong-etale-loc}, we can check if $\Z\subseteq\X$ is strong (resp.~topologically strong) \'etale locally. Thus, we need only understand the structure of those $\Z\subseteq\X=[U/G]$ with $U=\spec A$, where $\Z$ satisfies the conclusion of the lemma. 
Let $\Z$ be a reduced substack defined by an ideal $I\subseteq A$, let $\pi\colon\X\to U/G$ be the good moduli space map, and let $Z=\pi(\Z)$. The condition
that $\codim_\X \Z=\codim_\X Z$ is equivalent to $\hgt_A I=\hgt_{A^G} I^G =
\codim_\X \Z$. We have $\Z=\pi^{-1}(Z)$ if and only if $I^GA = I$, or equivalently $I=(f_1, \ldots , f_r)$ for some $f_i \in A^G$. Lastly, $\Z=(\pi^{-1}(Z))_{red}$ if and only if there is an auxiliary ideal $J = (f_1, \ldots , f_r)$ with $f_i \in A^G$ and $\sqrt{J} = I$.
\end{remark}

Using the notions of strong and topologically strong substacks, we obtain corresponding subgroups of Chow. The relative strong Chow groups were introduced in \cite{EdSa:16}.

\begin{definition}
  Let $\X \to X$ be a stable good moduli space morphism and assume
that $\X$ has pure dimension. Define the relative
  strong Chow group $A^k_{\st}(\X/X)$ to be the subgroup of $A^k(\X)$
  generated by the fundamental classes of strong integral substacks
  of codimension $k$. Likewise let $A^k_{\tst}(\X/X)$ be the subgroup generated by topologically strong integral substacks of codimension $k$. 
\end{definition}

\begin{remark}
  Unlike the Chow groups $A^k(\X)$, we see from the definition that both $A^k_{\st}(\X/X)$ and $A^k_{\tst}(\X/X)$ vanish 
  for $k > \dim \X$. Thus, the relative strong Chow group reflects some of the geometry of the good moduli space $X$.
  Moreover, if $\X$ is tame we also know that $A^k(\X)_\QQ$ is generated by fundamental classes of codimension-$k$ integral substacks of $\X$, as opposed to integral substacks on vector bundles over $\X$. Hence, $A^k_{\tst}(\X/X)_\QQ = A^k(\X)_\QQ \simeq A^k(X)_\QQ$.

However, the following example shows that the $A^k_{\st}(\X/X)$ need not equal
$A^k_{\tst}(\X/X)$ even after tensoring with $\QQ$.

\begin{example}[$A^k_{\st}(\X/X)_\QQ$ and $A^k_{\tst}(\X/X)_\QQ$ can differ]
\label{ex:tst-st-differ}
Consider $\GG_m$ acting on $\AA^4$ with weights $(1,-1,1,-1)$ and denote
by $x_1, x_2, x_3, x_4$ the coordinate functions on $\AA^4$. Let $\X = [\AA^4/\GG_m]$. Then the good moduli space is
$$X = \spec k[x_1x_2,x_1x_4, x_2x_3, x_3x_4].$$
If $\chi$ is the defining character of $\GG_m$ then the Chow group 
$A^1(\X) = \ZZ[t]$ where $t = c_1(\chi)$. If $\D \subseteq \X$ is an integral divisor, then since $\X$ is smooth, $\D$ is Cartier, hence of the form $[V(f)/\GG_m]$. Then by Remark \ref{rmk:local-structure-strong}, $\D$ is strong if and only if $f$ 
is a $\GG_m$-fixed polynomial. But such a polynomial has $\GG_m$-weight $0$ so its Chow class is
$0$. Hence $A^1_{\st}(\X/X) = 0$. On the other hand, the substack $\D'=[V(g)/\GG_m]$ with $g= x_1^2x_2 + x_3^2x_4$ is topologically strong and has non-trivial Chow class. To see that $\D'$ is topologically strong, note that if $I = (g)$ then $I^{\GG_m}k[x_1, x_2, x_3, x_4] = (x_4g, x_2g)$ and $V(x_4 g, x_2 g)$ is supported on $V(g)$ with an embedded component $V(x_2^2,x_4)$. Since the $\GG_m$-weight of $g$ is 1, we see $[\D']=t\neq0$. Hence $A^1_{\tst}(\X/X) = \ZZ t = A^1(\X)$.
\end{example}
\end{remark}

\subsection{Strong cycles and stratification by stabilizer}
Let $\X$ be a smooth connected Artin stack with good moduli space $\pi \colon\X \to X$. We show that $\pi$ induces a stratification of $X$ where each open stratum has finite quotient singularities. Let $\X_0$ be the locus of maximal dimensional stabilizer in $\X$ and let $X_0 \subseteq X$ be its image under $\pi$. By \cite{EdRy:17}, $\X_0$ is closed in $\X$ and thus $X_0 \subseteq X$ is as well. By \cite{Alp:13}, the restriction $\pi|_{\X_0}\colon\X_0 \to X_0$ is a good moduli space morphism. Now $\X_0$ is a smooth stack where the stabilizer dimension is constant.  By \cite[Proposition A.2]{EdRy:17} its rigidification is a smooth tame stack $\X^{rig}_0$ and the good moduli space morphism factors as $\X_0 \to \X_0^{rig} \stackrel{\eta}{\to} X_0$, where  
$\eta$ is a coarse space map. Hence $X_0$ has finite quotient singularities. Next, let $\Y_1$ be the locus of maximal dimensional stabilizer in $\X \smallsetminus \pi^{-1}(X_0)$. Again $\Y_1$ is a gerbe over a tame stack so its image $Y_1\subseteq X$ has finite quotient singularities. We let $\X_1=\overline{\Y}_1\cup\X_0$ which is closed, and let $X_1=\pi(\X_1)\subseteq X$ which is then closed as well. By upper semicontinuity of the dimension of stabilizers, we see $\overline{\Y}_1\setminus\Y_1\subseteq\X_0$, and so $\X_1\setminus\pi^{-1}(X_0)=\overline{\Y}_1\setminus\pi^{-1}(X_0)=\Y_1$. As a result, $X_1\setminus X_0=\pi(\X_1\setminus\pi^{-1}(X_0))=\pi(\Y_1)=Y_1$, which has finite quotient singularities. Continuing this process inductively we obtain a stratification $X_0 \subseteq X_1\subseteq \ldots \subseteq X_n = X$ where $X_{k+1}\smallsetminus X_k$ has finite quotient singularities. Note that $X \smallsetminus X_{n-1}$ is the image of $\ix^s$
under the good moduli space morphism.

The next proposition shows that topologically strong substacks of $\X$ must satisfy a transversality condition with respect to the above stratification.

\begin{proposition} \label{prop.stratify}

  Let $\pi\colon \X \to X$ be a properly
  stable good moduli space morphism
  and let $\Z$ be a topologically strong $k$-dimensional integral substack
  with image $Z\subseteq X$.
  Then $Z$ satisfies the following transversality conditions with respect to
  the stratification above.
  \begin{enumerate}
  \item\label{item::int-DM-locus} $Z \cap (X \smallsetminus X_{n-1}) \neq \emptyset$
  \item\label{item::int-lower-loci} For $k \leq n-1$,
$\dim \pi^{-1}(Z \cap Y) < \dim Z$ for every connected component $Y$ of $X_k$.
  \end{enumerate}
Conversely, if $Z \subseteq X$ is an integral closed subspace such
that $Z \cap (X \smallsetminus X_{n-1}) \neq \emptyset$, 
then let $\Z$ be the closure of $\pi^{-1}(Z \cap (X \smallsetminus X_{n-1}))$. If $\Z=\pi^{-1}(Z)_{red}$ then $\Z$ is topologically strong.
  \end{proposition}

\begin{proof}
  Since $\Z$ is saturated with respect to the good moduli space  morphism
  it must intersect the open substack $\ix^s$ because
  otherwise $\dim \pi(\Z) < \dim \Z$.
  Likewise if $Y$ is a connected component of $X_k$
  and $\dim \pi^{-1}(Z \cap Y) \geq \dim Z$ then $\pi^{-1}(Z)$ would have an additional irreducible component which does not dominate $Z$.

To prove the converse note that the good moduli space morphism $\pi$
is a homeomorphism over $X \smallsetminus X_{n-1}$, so
$\pi^{-1}(Z\cap(X\setminus X_{n-1}))$ is irreducible of the same
dimension as $Z$, and so its closure $\Z$ has the same property. By
assumption, $\Z$ is reduced so it is integral. Lastly, $\Z =
\pi^{-1}(Z)_{red}$ by assumption, so $\Z$ is topologically strong.
\end{proof}
\begin{remark}
The statement of Proposition \ref{prop.stratify} is easily modified when $\X \to X$ is stable
but not properly stable by adjusting our dimension counts by the dimension of the generic stabilizer of $\ix$.
\end{remark}
\begin{example} \label{ex:stratnecc}

Note that conditions (\ref{item::int-DM-locus}) and (\ref{item::int-lower-loci}) of Proposition \ref{prop.stratify}
are not sufficient for an integral subspace
$Z \subseteq X$ to be the image of a topologically strong cycle. For example, consider $\X = [\AA^4/\GG_m]$ where $\GG_m$ acts with weights $(1,-1,1,-1)$ on $\AA^4$ whose coordinates
are $(x_1, x_2, x_3, x_4)$. In this case the good moduli space of $\X$ is given by $X = \spec k[x_1x_2,x_1x_4,x_2x_3,x_3x_4]$ and the stratification is $\{O\} \subseteq X$ where $O$ is the origin corresponding to all coordinates equal to 0. In this case the integral Weil divisor $Z = V(x_1x_2,x_1x_4) \subseteq X$ satisfies the two conditions of the proposition.
However, the closure of $\pi^{-1}(Z \cap (X \smallsetminus \{O\}))$ is $[V(x_1)/\GG_m]$ while $\pi^{-1}(Z) = [V(x_1)/\GG_m] \cup [V(x_2,x_4)/\GG_m]$.

On the other hand, for a curve or point $Z \subseteq X$ to be the image of a topologically strong cycle then $Z$ must not contain the origin $O$, because the fiber of the good moduli space map $[\AA^4/\GG_m] \to X$ over $O$ is the union of two one-dimensional stacks
$[V(x_1,x_3)/\GG_m] \cup [V(x_2,x_4)/\GG_m]$. Hence $Z$ must be contained in
$X \smallsetminus \{O\}$. This condition is sufficient for $Z$ to the be the image of a topologically strong substack because the good moduli space
map $[\AA^4 \smallsetminus \{(0,0,0,0)\}/\GG_m] \to X \smallsetminus \{O\}$
is an isomorphism. 
\end{example}

The reasoning at the end of Example \ref{ex:stratnecc} can be extended to any
stable stack:
\begin{corollary} \label{cor:dim0or1}
With the notation as in Proposition \ref{prop.stratify}, if $\dim Z = 0,1$
then $Z$ is the image of a topologically strong cycle if and only if $Z \subseteq
X \smallsetminus X_{n-1}= \pi(\X^s)$.
\end{corollary}

We end with a lemma that will be used in the following section. Before we state the result, we establish some notation. 
Suppose $\X$ is an Artin stack with  stable good moduli space
$\pi \colon \X \to X$ and $\Z \subseteq \X$ is a strong (resp.~topologically) strong substack.
Let $Z = \pi(\Z)$. By \cite[Lemma 4.14]{Alp:13} the restriction
of $\pi$ to a morphism $\pi_\Z \colon \Z \to Z$ is a good moduli space morphism. Moreover, by Proposition \ref{prop.stratify} we know $\Z \cap \X^{s} \neq \emptyset$, so $\pi_\Z \colon \Z \to Z$ is also stable.

\begin{lemma} \label{lem:strongtransitivity}
  Let $\pi\colon \X \to X$ be an Artin stack with stable good moduli space
  and let $\Z$ be a strong (resp.~topologically strong) substack of $\X$.
  Let $\W \subseteq \Z$ be a strong (resp.~topologically strong) substack
  of $\Z$. Then $\W$ is a strong (resp.~topologically strong) substack
  of $\X$.
\end{lemma}
\begin{proof}
  First suppose that $\Z$ is strong in $\X$ and that $\W$ is strong in $\Z$.
  Let $Z = \pi(\Z)$, $W = \pi_{|\Z}(\W)$. By definition
  $\codim_\Z \X = \codim_Z X$, $\codim_\W \Z  = \codim_W Z$ and
$\Z = Z \times_X \ix$, 
  and $\W = W \times_Z \Z$.
Therefore $\codim_\W \X = \codim_W X$ and 
  $\W = W \times_X \ix$, so $\W$ is strong in $\X$.
  

  Now suppose $\Z$ is only topologically strong in $\X$ and that
  $\W$ is topologically strong in $\Z$. To show that $\W$ is topologically
  strong in $\X$ we need to show that $(W \times_X \X)_{red} = \W$. If we denote the fiber product $Z \times_X \X$ by $\Z'$, this is equivalent
  to showing that $(W \times_Z \Z')_{red} = \W$.

  Since $\Z$ is topologically strong in $\X$, we know that $\Z'_{red} = \Z$.
  We also know that $(W \times_Z \Z)_{red} = \W$ because $\W$ is topologically strong in $\Z$.
  Since the closed immersion of stacks $\Z \subseteq \Z'$ is a homeomorphism,
  the map obtained by base change $(W \times_Z \Z) \to (W \times_Z \Z')$
  is a closed immersion of stacks which is also a homeomorphism. It follows that they have the same reduced induced substack structure in $\X$.
  Therefore $\W$ is topologically strong in $\X$.
  \end{proof}

\section{The pullback of an operational class is topologically strong}
\label{sec:op-->strong}

In this section we prove part of  Theorem \ref{thm:inj-pullback}, namely that the pullback of operational classes are topologically strong. For this theorem we only require that $\pi \colon \X \to X$ is a stable good moduli space morphism.

\begin{theorem} \label{theorem.pullbacksarestrong}
  Let $\pi\colon\X \to X$ be a stable good moduli space morphism 
  with $\X$ irreducible. If $c \in A^k_{\op}(X)_\QQ$ then $\pi^*c \cap [\X]$ is represented by a topologically strong cycle.
\end{theorem}

\begin{remark} Note that Theorem \ref{theorem.pullbacksarestrong} holds without the assumption that the stack $\X$ is smooth.
\end{remark}
We begin with a series of lemmas, keeping the notation of the theorem throughout this section.

\begin{lemma} \label{lemma.c1}
  Let $\X$ be an integral stack and let
  $\pi\colon\X \to X$ be a stable good moduli space morphism.
  If $X$ has an ample line bundle and $L$ is any line bundle on $X$,
  then $c_1(\pi^*L)\cap [\X]$ is a strong cycle.
\end{lemma}
\begin{proof}
Since $X$ has an ample line bundle we can write $L = L_1 \otimes L_2^{-1}$ with
$L_1$ and $L_2$ very ample, see e.g.~\cite[Exercise II.7.5]{Hartshorne}. 
Thus, it suffices to prove that $c_1(\pi^*L) \cap [\X]$ is strong when 
$L$ is very ample.
Let $\X^{s}$ be the open set of stable points and let $X^s = \pi(\X^s)$ be its good moduli space. The dimension of the stabilizers is constant
on $\X^s$ and so it is a gerbe over a tame stack $(\X)^s_{\tame}$
whose coarse space
is $X^s$. Since both a gerbe morphism and a coarse moduli space morphism
are homeomorphisms of Zariski spaces 
the good moduli space morphism $\pi_{|_{\X^s}}$ induces a homeomorphism
$|\X^s| \to |X^s|$.

If $\dim X = 1$, then we can choose a divisor $D = P_1 + \ldots + P_d$ 
where $P_1, \ldots , P_d$ are distinct points contained in the 
intersection $X^s \cap X^{sm}$ such that $L = L(D)$; here $X^{sm}$ denotes the smooth locus of $X$, which is non-empty because $X$ is reduced and we work over an algebraically closed field. 
In this case the inverse image of each $P_i$ is a strong stacky point ${\mathcal P}_i$ in $\X$
and so $c_1(\pi^*L) \cap [\X] = [{\mathcal P}_1] + \ldots  + [{\mathcal P}_d]$ 
is strong.

Next suppose $\dim X > 1$. Since $\X^s$ is dense in $\X$,
its complement contains finitely many (possibly zero) divisors $\D_1,\dots,\D_r$ of
$\X$. Let $W_i=\pi(\D_i)$. Since $L$ is very ample we can choose a
Cartier divisor $D$ with $|D|$ integral such that $L = L(D)$. Moreover, we can
choose $D$ so that $D\cap X^s\neq\varnothing$ and $D$ does not contain
any of the $W_i$. The class $c_1(\pi^*L) \cap [\X]$ is represented by
the Cartier divisor $\pi^{-1}(D)$. Let $\D$ be the closure of
$\pi^{-1}(D\cap X^s)$.
Since  the map $|\X^s| \to |X|$ a homeomorphism.
$\D = \overline{\pi^{-1}(D \cap X^s)}$
is irreducible because $\pi^{-1}(D \cap X^s)$ is irreducible.

We claim that $\D=\pi^{-1}(D)$. Upon showing
this, we are done because $\pi^{-1}(D\cap X^s)$ is integral and has the
same dimension as $D$, so the same is true for $\D$. It follows that
$\pi^{-1}(D)=\D$ is strong, so $c_1(\pi^*L) \cap [\X]$ is as well.

It remains to prove the claim that $\D=\pi^{-1}(D)$. First note that
$\D\subseteq\pi^{-1}(D)$, and moreover
$\D\cap\X^s=\pi^{-1}(D)\cap\X^s$. So if $\D$ and $\pi^{-1}(D)$ do not have
the same support, $\pi^{-1}(D)$ must contain an irreducible component
in $\X\setminus\X^s$. Since $\pi^{-1}(D)$ is a Cartier divisor, it would
have to contain a divisorial component of $\X\setminus\X^s$, but this is
not possible since $D$ does not contain any of the $W_i$. As a result,
$\D_{red}=\pi^{-1}(D)_{red}$.
As $\pi^{-1}(D)$ is Cartier, it has no
embedded components. Thus $\pi^{-1}(D)$ agrees with
$\D$ up to multiplicity since they both have the same support and this support is irreducible.
Since $\D$ is reduced and $\pi^{-1}(D)\cap
\X^s=\D\cap \X^s$, these multiplicities agree and are equal to 1. 
 \end{proof}
\begin{remark}
If $\X$ is not reduced, the proof of Lemma \ref{lemma.c1}
still shows that $c_1(\pi^*L) \cap [\X]$ is topologically strong.
\end{remark}

\begin{proposition} \label{prop.pushforward}
  Let
  \begin{equation} \label{diag:cartesianpush}
\xymatrix{ \X' \ar[r]^f \ar[d]^-{\pi'} & \X \ar[d]^{\pi}\\
    X' \ar[r]^g  & X }
\end{equation}
  be a cartesian diagram. We assume $\pi$ and $\pi'$ are 
  stable good moduli space morphisms, $\X'$ and $\X$ are irreducible,
 and $f$ and $g$ are proper.
  Then for all $\alpha \in A^*_{\tst}(\X'/X')$, the proper pushforward $f_*\alpha$ is contained in $A^*_{\tst}(\X/X)$.
\end{proposition}

\begin{proof}
We can assume that $\alpha = [\Z']$ where $\Z'$ is a topologically strong integral substack. Let $\Z = f(\Z')$ with its reduced substack structure. We may also assume that $\dim \Z = \dim \Z'$ since otherwise
  $f_*[\Z'] = 0$.
  
Since $\Z'$ is topologically strong the same argument used in
the proof of Proposition \ref{prop.stratify} in the properly stable
case shows that $\Z' \cap (\X')^{s} \neq \emptyset$, where $(\X')^{s}$.

{\bf Claim:}
$$f^{-1}(\X^{s}) = (\X')^{s}.$$

If $\x'$ is any closed point of $\X$ set $\x = f(\x')$, $x' = \pi^\prime(\x')$, and  $x = \pi(\x)$. Then we have the following diagram of stacks and good moduli spaces where all squares are cartesian.

\begin{equation} \label{diag:fibers}
\xymatrix{ \pi'^{-1}(x') \ar[d] \ar@{^{(}->}[r] & \pi'^{-1}(f^{-1}(\x))  
\ar[d] \ar[r]^-{f} & \pi^{-1}(x) \ar[d]\\
x' \ar@{^{(}->}[r] & g^{-1}(x)  \ar[r]^-{g} & x}
\end{equation}

Now if $\x' \in f^{-1}(\X^{s})$ then $|\pi^{-1}(x)| =\x$.
By base change
along the
finite morphism of $x' \to x$, we know that
$\pi'^{-1}(x') \to \pi^{-1}(x)$ is a finite morphism. Since
$|\pi^{-1}(x)|$ is a single point, it follows that $|\pi'^{-1}(x')|$
is discrete.
However, we also know that $|\pi'^{-1}(x')|$ has a unique closed point,
so it must be a singleton. Hence $\x'$ is saturated. 


Now suppose $\x' \in (\X')^{s}$. We wish to show that $\x = f(\x')$
is also  stable; i.e., 
that $|\pi^{-1}(\pi(\x))|$ consists of the single point $\x$.

Again we refer to diagram \eqref{diag:fibers}.
The map of points $x' \to x$ is surjective, so
$|\pi^{-1}(x)| \to  |\pi'^{-1}(x')|$ is surjective as well.
By assumption that $\x'$ is  stable,
$\pi^{-1}(x')$ consists of a single point, so therefore $|\pi^{-1}(x)|$ does as well. Hence, $f((\X')^{s}) \subset \X^{s}$. This proves our claim.
\qed

Given the claim it follows that
$\Z = f(\Z')$ has non-empty intersection with the open substack $\X^{s}$. 
Thus $\codim_X \pi(\Z) = \codim_\X \pi(\Z \cap \X^{s}) = \codim_\X \Z$. 

Let $Z = \pi(\Z)$ and let $X^{s} = \pi(\X^{s})$. Since $\X^{s} \to X^{s}$
factors as gerbe morphism followed by a coarse moduli space map for a tame stack, we know that $\pi^{-1}(Z \cap X^{s})_{red} = \Z \cap \X^{s}$. Thus, to 
prove that
$\pi^{-1}(Z)_{red} = \Z$ it suffices to show that $\pi^{-1}(Z)$ is
irreducible.

Since $\Z'={\pi'}^{-1}(Z')_{red}$ we see $\Z'$ and $\pi'^{-1}(Z')$ 
have the same image  under the morphism
$f$. Hence, we can replace $X'$ by $Z'$ and $\X'$ by $\Z'$. That is, we may assume that $f,g$ are surjective 
and prove that $\pi^{-1}(g(X'))$ is irreducible.

Since the diagram is cartesian (and thus commutative) 
the map $f \colon \X' \to \X$ factors
through the closed substack $\pi^{-1}(g(X')) \subseteq \X$. Thus,
by the universal property of scheme-theoretic images, we have a closed immersion $i\colon f(\X')\to \pi^{-1}(g(X'))$. We therefore have a commutative diagram
\[
\xymatrix{
\X'\ar@{->>}[r]^-{f}\ar[d]^-{h} & f(\X')\ar@{^{(}->}[r]^-{i} & \pi^{-1}(g(X'))\ar@{^{(}->}[r]\ar[d] & \X\ar[d]^\pi\\
X'\ar@{->>}[rr]^-{g} & & g(X')\ar@{^{(}->}[r] & X
}
\]
Since the righthand square is cartesian, and the outer rectangle is
cartesian, we see the left-hand rectangle is cartesian as well.
Therefore, the composite $ i \circ f$ is surjective by base change.
Hence $i$ is a surjective closed embedding. Hence $f(\X') = \pi^{-1}(g(X'))$
as closed subsets of $\X$. Therefore, $\pi^{-1}(g(X'))$ is irreducible.
\end{proof}

\begin{proposition} \label{prop.strongchernclasses}
  Let $\pi \colon \X \to X$ be a stable good moduli space morphism with $X$ a separated algebraic space and $\X$ irreducible.
If $E$ is a vector bundle on $X$, then $c_i(\pi^*E) \cap [\X]$ is a topologically strong cycle.
More generally, if $q(E_1, \ldots , E_r)$ is a polynomial in the Chern classes
of vector bundles $E_1, \ldots , E_r$ then $\pi^*q(E_1, \ldots , E_r) \cap [\X]$ is a topologically strong cycle. 
\end{proposition}
\begin{proof}
  We first
 use induction to
  prove the special case that $c_1(\pi^*L)^m \cap[\X]$ is topologically strong
  for all line bundles $L$, all exponents $m$, and all stable good moduli space morphisms $\pi\colon\X\to X$ with $X$ a separated algebraic space.
We begin with the base case $m=1$. By Chow's lemma \cite[Tag 089K]{stacks-project}, there is a scheme $X'$ which has an ample line bundle and a proper morphism $g \colon X'\to X$ which is an isomorphism over a dense open set. Consider the cartesian diagram
$$\xymatrix{\X' \ar[d]^{\pi'}\ar[r]^f & \X \ar[d]^{\pi} \\ X'
    \ar[r]^g & X}$$

If $L$ is any line bundle on $X$,
then 
Lemma
\ref{lemma.c1} shows that that $c_1(\pi'^*L')\cap [\X']$ is represented by
a strong cycle on $\X'$ where $L' = g^*L$.
Then by
Proposition \ref{prop.pushforward}, we that know $f_*(c_1(\pi'^*L') \cap
[\X'])$ is topologically strong.
Since $\pi'^*L'=f^*\pi^*L$, by the
projection formula $f_*(c_1(\pi'^*L') \cap [\X']) =
c_1(\pi^*L) \cap f_*[\X'] = c_1(\pi^*L) \cap [\X]$.

We now show $c_1(\pi^*L)^{m+1} \cap[\X]$ is topologically strong
assuming the result for $m$. By
assumption,
we can write $c_1(\pi^*L)^{m} \cap [\X] = \sum a_i [\Z_i]$ where the $\Z_i$
are topologically strong
substacks of $\X$.
Let $Z_i=\pi(\Z_i)$ and $\pi_i=\pi|_{\Z_i}$.
By \cite[Lemma 4.14]{Alp:13}  $\pi_i\colon\Z_i\to Z_i$ is a good moduli space. Since $\Z_i \cap \X^{s} \neq \emptyset$ we know
that $\pi_i$ is also a stable good moduli space morphism.
By the case $m =1$ of our statement we know that, $\pi_i^* L|_{Z_i} \cap [\Z_i]$
is a represented by a topologically strong cycle on
$\Z_i$.

Now note that
\[
c_1(\pi^*L)^{m+1} \cap[\X]=\sum_i a_i\, c_1(\pi^*L)\cap[\Z_i]=\sum_i a_i\, c_1(\pi_i^*L|_{Z_i})\cap[\Z_i].
\]
Since $c_1(\pi_i^*L|_{Z_i})\cap[\Z_i]$ is topologically strong on
$\Z_i$, it is also topologically strong on $\X$ by Lemma \ref{lem:strongtransitivity}.
  This proves  the result
for $c_1(\pi^*L)^{m+1} \cap[\X]$.

Now let $E$ be a vector bundle of rank $r$ on $X$. We show that $c_i(\pi^*E) \cap [\X]$ is topologically strong using the construction of Chern classes via Segre classes, see \cite[Chapter 3]{Fulton}.\footnote{One may alternatively proceed by using the splitting principle.} Consider the cartesian diagram
$$\xymatrix{ \Proj(\pi^*E) \ar[d]^\eta \ar[r]^-{f} & \X \ar[d]^\pi\\
\Proj(E) \ar[r] & X}$$
Since we have already established the proposition for powers of the first Chern class for all good moduli space morphisms, applying this to $\eta$ shows that $c_1(\eta^*{\mathcal O}_{\Proj(E)(1)})^{i+r - 1} \cap [\Proj(\pi^*E)]$ is topologically strong. Then by Proposition \ref{prop.pushforward}, the Segre classes
\[
s_i(\pi^*E) \cap [\X]= f_*(c_1(\eta^*{\mathcal O}_{\Proj(E)}(1))^{i+r - 1} \cap [\Proj(\pi^*E)])
\]
are topologically strong. Since the Chern classes of $E$ are polynomials in the Segre classes, the proposition follows by induction on the degree of the polynomial in the same manner as the argument for powers of the first Chern class.
\end{proof}


\begin{proof}[Proof of Theorem \ref{theorem.pullbacksarestrong}]
By Chow's lemma and resolution of singularities\footnote{This proof also works in positive characteristic if we use a smooth alteration instead of
a resolution of singularities.} there is a smooth scheme $X'$ with ample line bundle and a proper
  morphism $g \colon X' \to X$ which is an isomorphism on a dense open
  set. Consider the cartesian diagram
  \[
  \xymatrix{
  \X'\ar[r]^f\ar[d]^-{\pi'} & \X\ar[d]^-{\pi}\\
  X'\ar[r]^g & X
  }
  \]
  Let $c  \in A^*_{\op}(X)_\QQ$.  By the Riemann-Roch theorem for smooth schemes, $A^*_{\op}(X')_\QQ$ is generated by Chern classes of vector bundles. Then $g^*c$ is a rational polynomial in Chern classes of vector bundles on $X'$, so by Proposition \ref{prop.strongchernclasses}, $\pi'^*g^*c \cap [\X']$ is topologically strong. Then Proposition \ref{prop.pushforward} says $f_*(\pi'^*g^*c \cap [\X'])$ is topologically strong. Lastly, notice by the projection formula that $f_*(\pi'^*g^*c \cap [\X']) = f_*(f^*\pi^*c \cap [\X'])=\pi^*c\cap [\X]$.
\end{proof}


\section{Injectivity of the pullback and applications}
\subsection{Proof of Theorem \ref{thm:inj-pullback}}
\label{sec:inj-pullback}
To complete the proof of the theorem we need to show that $\pi^*$ is injective
when $\pi \colon \X \to X$ is properly stable.

\begin{proof}[Proof of the injectivity of $\pi^*$.]
  We use induction on $n$ which is the maximum dimension of a stabilizer
  of a point of $\X$. When $n=0$ then $\X$ is tame and the pullback
  map $\pi^* \colon A^*_{\op}(X)_\QQ \to A^*(\X)_\QQ$ is an isomorphism.

  Next, suppose the theorem hold for stacks with maximum dimension
  of a stabilizer of a point less than or equal to $n-1$. Let $\Y \subseteq \X$ be the locus with stabilizer dimension $n$
  and let $f \colon \X' \to X$ be the Reichstein transformation along $\Y$. By \cite{EdRy:17} we know that the maximum dimension of the stabilizer of a point of $\X'$ is less than $n$. Let $\pi'\colon\X'\to X'$ be a good moduli space morphism and consider the induced commutative diagram:
    $$\xymatrix{\X' \ar[d]^-{\pi'} \ar[r]^f & \X \ar[d]^\pi\\
    X' \ar[r]^g & X}$$
Since $g \colon X' \to X$ is proper and surjective, the pullback $g^* \colon A^*_{\op}(X)_\QQ
  \to A^*_{\op}(X')_\QQ$ is injective by \cite[Lemma 2.1]{Kim:90}. Also by our induction hypothesis we know that
  $(\pi')^* \colon A^*_{\op}(X')_\QQ \to A^*(\X')_\QQ$ is injective. By commutativity
  of the diagram we conclude that $f^* \pi^*$ is injective, hence $\pi^*$
  is injective.
\end{proof}

\subsection{Applications to Picard groups of good moduli spaces: proof of Theorem \ref{thm:pic-surjectivity}}
The proof Theorem \ref{thm:pic-surjectivity} makes essential use of the following result.
\begin{theorem}
\label{thm:same-images}
Let $\pi\colon\X\to X$ be a properly stable good moduli space morphism with $\X$ smooth, and consider the commutative diagram
\[
\xymatrix{
\Pic(\X)\ar[r]^{\simeq} & A^1(\X)\\
\Pic(X)\ar[r]\ar[u]^{\pi^*} & A^1_{\op}(X)\ar[u]^{\pi^*}
}
\]
Then the images of $\Pic(X)$ and $A^1_{\op}(X)$ in $A^1(\X)$ agree.
\end{theorem}
\begin{proof}
Let $c\in A^1_{\op}(X)$. Then $\pi^*c\in A^1(\X) \simeq\Pic(\X)$ so it is represented by a line bundle $\L$ on $\X$. Showing that $\L$ is the pullback of a line bundle from $X$ is equivalent to showing that the stabilizer actions on $\L$ are trivial. Let $x$ be a point of $\X$ with stabilizer $G_x$. The map from the residual gerbe $i\colon BG_x\to\X$ induces a map of good moduli spaces making the diagram
\[
\xymatrix{
BG_x\ar[r]^-{i}\ar[d]^-{\eta} & \X\ar[d]^-{\pi}\\
\spec k\ar[r]^-{j} & X
}
\]
commute. Then $i^*\L$ is a line bundle on $BG_x$ corresponding to a character $\chi$ of $G_x$. The stabilizer action of $G_x$ on $\L$ is trivial if and only if $\chi$ is the trivial character. Note that $j^*c\in A^1_{\op}(\spec k)=0$ and so $j^*c$ is represented by the trivial line bundle. Since $i^*\L=\eta^*j^*c$, we see $i^*\L$ is also the trivial line bundle, i.e.~it corresponds to the trivial representation of $G_x$, and hence the character $\chi$ is trivial.
\end{proof}

\begin{proof}[Proof of Theorem \ref{thm:pic-surjectivity}]
  Let $\alpha\colon\Pic(X)\to A^1_{\op}(X)$ denote the natural map. Given an element $c\in A^1_{\op}(X)$, by Theorem \ref{thm:same-images} there exists $L\in\Pic(X)$ such that $\pi^*L=\pi^*c$ in $A^1(\X)$. That is, $\pi^*\alpha(L)=\pi^*c$. Since $\pi^*\colon A^1_{\op}(X)_\QQ\to A^1(\X)_\QQ$ is injective by
  Theorem \ref{thm:inj-pullback}, we see $\alpha(L)=c$, and so $\alpha$ is surjective. Injectivity follows from Remark \ref{rem.injective}.
\end{proof}

\section{Surjectivity results for the pullback map}
\label{sec:surj-pullback}

\subsection{The case of a moduli space with finite quotient
  singularities: proof of Theorem \ref{thm:surj-pullback}(\ref{surj::quot-sing})}
Let $\Z$ be a topologically strong integral substack of codimension-$k$ in $\X$ with image $Z \subseteq X$. Since $X$ has finite quotient singularities, $A^*(X)_\QQ = A^*_{\op}(X)_\QQ$. Let $c_Z$ be the bivariant class in $A^k(Z \to X)$ corresponding to the refined intersection with $Z$. In other words, given a map $T \to X$ and $\alpha \in A_*(T)$
we set $c_Z \cap \alpha = [Z] \cdot \alpha \in A_{*-k}(Z \times_X T)$.
Since $Z \to X$ is proper there is a pushforward $A^k(Z \to X)_\QQ \to A^k_{\op}(X)_\QQ =A^k(X)_\QQ$
\cite[17.2 P2]{Fulton} and the image is the fundamental class of $Z$. 

To compute the pullback of $[Z]$ we can compute $j_*(c_Z \cap [\X])$
where $j \colon \Z \to \X$ is the inclusion morphism. Now $c_Z \cap [\X]$
is a cycle of dimension equal to $\dim \X - k=\dim X - k=\dim Z$ supported on the fiber
product $Z \times_X \X$. Since $\Z$ is assumed to be topologically strong,
$(Z \times_X \X)_{red} = \Z$. Hence, $c_Z \cap [\X]$ is a multiple
of the fundamental class in $A^0(\Z)$. This multiple is necessarily non-zero
because it is non-zero when we restrict to the non-empty open subset 
$\X^{ps} \cap Z$. Hence the pullback of $[Z]$ in $A^*(\X)$ is a non-zero multiple of $[\Z]$. In other words $[\Z]$ is in the image of $A^k_{\op}(X)_\QQ = A^k(X)_\QQ$.

\subsection{The case of codimension-one cycles: proof of Theorem \ref{thm:surj-pullback}(\ref{surj::A1})}
By Theorem \ref{thm:same-images} we know that the image of $A^1_{\op}(X)$ in
$A^1(\X)$
is the same as the image of $\Pic(X)$ under the map $L \mapsto c_1(\pi^*L) \cap [\X]$. Since $X$ has an ample line bundle we know by Lemma \ref{lemma.c1} that
$c_1(\pi^*L) \cap [\X]$ is strong. Hence the image of $A^1_{\op}(X)$ is contained in $A^1_{\st}(\X/X)$. Conversely, if $\Z \subseteq \X$ is a strong codimension-one
substack, then $[\Z]$ is an effective Cartier divisor whose local equation
in a neighborhood of every point $x$ is $G_x$-fixed, where $G_x$ denotes the stabilizer of $x$. Hence, the
restriction of ${\mathcal O}_\X(\Z)$ to $BG_x$ is the trivial character. Thus, ${\mathcal O}_\X(\Z)$ is the pullback of a line bundle on $X$, so $\pi^* A^1_{\op}(X) = A^1_{\st}(\X/X)$.

\subsection{Strong regular cycles: proof of Theorem \ref{thm:surj-pullback}(\ref{surj::reg-embedded})}
\label{subsec:surj-regular-strong}

Let $\Z \subseteq \X$ be strong and regularly embedded. By \cite[Proposition 1.14]{EdSa:16}, 
we have a cartesian diagram
\[
\xymatrix{
\Z\ar[r]\ar[d] & \X\ar[d]\\
Z\ar[r] & X
}
\]
and $Z\subseteq X$ is regularly embedded. Letting
$K_{\bullet}(I_Z/I_Z^2)$ be the Koszul resolution of the conormal
bundle, we have an equality $[{\mathcal O}_Z]=
[K_{\bullet}(I_Z/I_Z^2)]$ in $K$-theory. Thus ${\mathcal O}_Z$ has a
Chern character and $[Z] = ch({\mathcal O}_Z) \cap [X]$. Similarly,
${\mathcal O}_\Z$ has a Chern character and $[\Z] = ch({\mathcal
  O}_\Z) \cap [\X]$. To finish the proof
notice that $ch({\mathcal O}_\Z) \cap [\X]=ch(\pi^*{\mathcal O}_Z) \cap [\X]$ which is the pullback of
    an operational class.\qed

\subsection{The case of small dimension: proof of Theorem \ref{thm:surj-pullback}(\ref{surj::dimX}) and Corollary \ref{cor:lowdim}}

Let $\Z \subseteq \X$ be a topologically strong integral substack of dimension 0
or 1 and let $Z \subseteq X$ be its image in $X$. By Corollary \ref{cor:dim0or1} 
we know that $Z \subseteq X^{ps}$ where $X^{ps}$ is the image
of the maximal saturated tame stack $\X^{s}$ and moreover, 
$A^*_{\op}(X^{ps})_\QQ = A^*(X^{ps})_\QQ =
A^*(\X^{ps})_\QQ$. Let $c_Z \in A^*(Z \to X^{ps})_\QQ$ be the bivariant class corresponding to the
refined intersection with $Z$ in $X^s$. In other words, given a map $T \to X^s$ and $\alpha \in
A^*(T)$, we define $c_Z \cap \alpha = [Z] \cdot \alpha \in A_*(Z \times_{X^s} T)$. As was the case in the proof of Theorem \ref{thm:surj-pullback}(\ref{surj::quot-sing}), $c_Z \cap [\X^{ps}]$ is necessarily a non-zero multiple of the fundamental class of $\Z= \pi^{-1}(Z)_{red}$.

Since $Z\subseteq X$ is closed, $c_Z$ has a canonical lift to a
bivariant class $A^*_{\op}(Z \to X)$ defined by the formula $\alpha \mapsto Z \cdot \alpha|_{X^s \times_X T}$ and we let $c$ be the direct image in $A^k_{\op}(X)$. Since $c_Z \cap [\X^{ps}]$ is a non-zero multiple of $[\Z]$ in $A^k(\X^{ps})$, it follows that $c \cap [\X]$ is a non-zero multiple of $[\Z]$ in $A^k(\X)$. Hence, $[\Z]$ is in the image of $A^k_{\op}(X)_\QQ$. Finally if $\X^{ps}$ is representable then $\X^{ps} = X^{ps}$ so $\Z = Z$ is automatically strong.

Corollary \ref{cor:lowdim} is now a consequence of what has already been proved.
Clearly $A^0(\X)_\QQ = A^0_{\st}(\X/X) = A^0(X) = \QQ$ for any integral
stack $\X$.
If $X$ has an ample line bundle then $A^1_{\st}(\X/X)_\QQ  = A^1_{\op}(X)_\QQ$ by Theorem \ref{thm:surj-pullback}(\ref{surj::A1}). 
Since $\dim X \leq 3$, $A^2_{\st}(\X/X)_\QQ = A^2_{\op}(X)_\QQ$ and
$A^3_{\st}(\X/X)_\QQ = A^3_{\op}(X)_\QQ$ by Theorem \ref{thm:surj-pullback}(\ref{surj::dimX}).

\begin{example}[The strong Chow ring of the toric stack associated
    to the projectivized quadric cone]
  We consider the strong Chow ring of the toric stack $\X$ associated to the
  projectivized quadric cone which was studied in \cite[Example 3.12]{EdSa:16}. There we computed the integral strong Chow ring modulo a conjecture about $A^2_{\st}(\X/X)$. Using Corollary \ref{cor:lowdim} we can verify this conjecture without qualification.

  Precisely, we let $\X = [U/\GG_m^2]$ where $U = \AA^5 \smallsetminus V(x_1x_2, x_1x_4,x_2 x_3, x_3x_4,v)$ and $\GG_m^2$ acts diagonally on $\AA^5$ with weight matrix
$$\left(\begin{array}{ccccc} 1 & 0 & 1 & 0 & 1\\0 & 1 & 0 & 1 & 1\end{array}
    \right).$$
    Here we have coordinates $(x_1, x_2, x_3, x_4, v)$ on $\AA^5$. Then the good moduli space of $\X$ is given by $X = \Projj k[x_1x_2, x_1x_4, x_2x_3,x_3x_4,v]$ which is the projective closure in $\Pro^4$ of the cone over the quadric surface in $\Pro^3$. In \cite[Example 3.12]{EdSa:16} we showed
    that $A^*(\X) = \ZZ[s,t]/(s^2(s+t), t^2(s+t))$ where $s,t$ are the first Chern classes of the characters of $\GG_m^2$ corresponding to projection onto the first or second factors. We also showed that $A^1_{\st}(\X/X)$ is the free abelian group generated by $s+t$, and that $A^3_{\st}(\X/X)$ is the free abelian group
    generated by $st (s+t)$. We showed that $A^2_{\st}(\X/X)$ contains the free group of rank two generated by $s(s+t)$ and $t(s+t)$ and conjectured that this was the entire relative strong Chow group in codimension 2. Since
    $X$ is projective, by Corollary \ref{cor:lowdim} we know that $\rk A^2_{\op}(X) = \rk A^2_{\st}(\X/X)$.
    On the other hand, $X$ is a toric variety so by \cite{FuSt:97} we know that $A^2_{\op}(X) = \hom(A_2(X),\ZZ)$. Since $X$ is toric, it is easy to calculate by hand that the group $A_2(X)$ has rank 2. Thus $\rk A^2_{\st}(\X/X) = 2$. Moreover, the
    lattice $\langle s(s+t), t(s+t) \rangle$ is primitive in the rank 3
    free abelian group $A^2(\X)$, so we conclude $A^2_{\st}(\X/X) = \langle
    s(s+t), t(s+t) \rangle$

    Note that the classes $s(s+t)$, and $t(s+t)$ which correspond to distinct operational classes are represented by the cycles $[V(x_1,v)/\GG_m]$ and
    $[V(x_2,v)]/\GG_m$.
    These cycles have rationally equivalent images in $A^2(X)$. This illustrates the power of viewing the singular toric variety $X$ together with the smooth
    quotient stack $\X$.
    \end{example}

\section{Conjectures and speculations}
\label{sec:speculations}
We end this paper with some questions and conjectures about the image
of $A^*_{\op}(X)$ in $A^*(\ix)$.

Theorem \ref{thm:inj-pullback} states that $A^*_{\op}(X)_\QQ$
injects into $A^*_{\tst}(\ix/X)_\QQ$.
However, we know that the image need not equal $A^*_{\tst}(\ix/X)_\QQ$.
This follows from
Example \ref{ex:tst-st-differ} where we constructed a topologically strong
Cartier divisor that is not strong. Since the image of $A^1_{\op}(X)$ equals $A^1_{\st}(\ix/X)$, (Theorem \ref{thm:surj-pullback}(b)) we see that not every topologically strong cycle is a pullback of an operational class.

This of course raises the question: what is the image of $\pi^*$? We state a conjectural answer after first introducing some definitions.

%


\begin{definition}
Let $\X$ be an Artin stack and $\pi\colon\X\to X$ a stable good moduli space. We say an integral substack $\Z\subseteq\X$ is \emph{primarily strong} if $\Z$ is topologically strong and $\pi^{-1}(\Z)$ has no embedded components. If $\X$ has pure dimension, then we define $A^k_{\pst}(\X/X)$ to be the subgroup of $A^k(\X)$ generated by primarily strong integral substacks of codimension $k$.
\end{definition}

\begin{remark}
The proof of Lemma \ref{l:strong-etale-loc} also gives an \'etale local characterization of is primarily strong substacks. As in Remark \ref{rmk:local-structure-strong}, we see that primarily strong substacks are those that \'etale locally have the following form: $\Z\subseteq\X=[U/G]$ with $G$ the stabilizer of a point, $U=\spec A$, and $\Z$ is defined by an ideal $I\subseteq A$ with $\dim A/I=\dim A^G/I^G$ and such that there is an $I$-primary ideal $J$ satisfying $J=(f_1,\dots,f_r)$ with each $f_i\in A^G$.
\end{remark}

\begin{conjecture}
\label{conj:image-pi-pst}
Let $\X$ be a smooth connected properly stable Artin stack with good moduli space $\pi\colon\X\to X$. Then the injection $\pi^*\colon A^*_{\op}(X)_\QQ\to A^*_{\tst}(\X/X)_\QQ$ induces an isomorphism $A^*_{\op}(X)_\QQ\simeq A^*_{\pst}(\X/X)_\QQ$.
\end{conjecture}

We prove Conjecture \ref{conj:image-pi-pst} in several cases by giving a minor strengthening of Theorem \ref{thm:surj-pullback}.

\begin{theorem}
\label{thm:conj-pst}
Let $\X$ be a properly stable smooth Artin stack with good moduli space $\pi\colon\X\to X$. Then the following hold:
  \begin{enumerate}[label=\emph{(\alph*)}, ref=\alph*]
    \item \label{pst::quot-sing}
Conjecture \ref{conj:image-pi-pst} holds if $X$ is smooth.

  \item \label{pst::A1} If $X$ has an ample line bundle, then Conjecture \ref{conj:image-pi-pst} holds in codimension 1.

\item \label{pst::dimX}
Assume $k$ equals $\dim X -1$ or $\dim X$. If the maximal saturated tame substack of $\X$ is representable, then Conjecture \ref{conj:image-pi-pst} holds in codimension $k$.
\end{enumerate}
\end{theorem}
\begin{proof}
We begin with case (\ref{pst::quot-sing}). By Theorem \ref{thm:surj-pullback}(\ref{surj::quot-sing}), it suffices to prove that $A^*_{\tst}(\X/X)_\QQ=A^*_{\pst}(\X/X)_\QQ$.   Since $\X$ and $X$ are smooth, $\pi$ is an lci
  morphism. Furthermore, it has relative dimension 0 since the maximal
  saturated tame substack $\U\subseteq\X$ is non-empty and
  $\U\to\pi(\U)$ has relative dimension 0. Thus, $\pi$ locally factors
  as $\X\stackrel{j}{\to}\Y\stackrel{p}{\to} X$ with $j$ a regular
  embedding of codimension $e$ and $p$ a smooth morphism of relative
  dimension $e$. Let $\Z\subseteq\X$ be a topologically strong
  substack of dimension $d$ and let $Z=\pi(\Z)$. By hypothesis, $\dim
  Z=\dim\Z=d$. Then $p^{-1}(Z)$ has dimension $d+e$ and
  $\pi^{-1}(Z)=j^{-1}p^{-1}(Z)$ is cut out by $e$ equations within
  $p^{-1}(Z)$. As a result, each component (irreducible or embedded)
  of $\pi^{-1}(Z)$ has dimension at least $d$. Since $\Z$ is
  topologically strong, $\Z=\pi^{-1}(Z)_{red}$ and so $\pi^{-1}(Z)$
  has exactly one irreducible component of dimension $d$, As a result,
  $\pi^{-1}(Z)$ has no embedded components since such components would
  have dimension less than $d$.

To handle case (\ref{pst::A1}), it suffices by Theorem \ref{thm:surj-pullback}(\ref{surj::A1}) to show that $A^1_{\st}(\X/X)_\QQ=A^1_{\pst}(\X/X)_\QQ$. Let $\Z\subseteq\X$ be a primarily strong divisor. Since $\X$ is smooth, $\Z$ is Cartier, and so some multiple of its local equation is invariant under the stabilizer actions. As a result, $n[\Z]=c_1(\pi^*\L)\cap[\X]$ for some line bundle $\L$ on $X$. Then Lemma \ref{lemma.c1} tells us that $n[\Z]$ is strong.

Lastly, we turn to case (\ref{pst::dimX}). By the first assertion of Theorem \ref{thm:surj-pullback}(\ref{surj::dimX}), we need only show $A^k_{\pst}(\X/X)_\QQ = A^k_{\tst}(\X/X)_\QQ$. Since $A^k_{\st}(\X/X)_\QQ\subseteq A^k_{\pst}(\X/X)_\QQ\subseteq A^k_{\tst}(\X/X)_\QQ$, the second assertion of Theorem \ref{thm:surj-pullback}(\ref{surj::dimX}) shows that $A^k_{\st}(\X/X)_\QQ = A^k_{\pst}(\X/X)_\QQ = A^k_{\tst}(\X/X)_\QQ$.
\end{proof}

\begin{remark} One way to partially verify Conjecture \ref{conj:image-pi-pst} is to generalize Proposition \ref{prop.pushforward}  by proving that if
  $$\xymatrix{ \X' \ar[d] \ar[r] & \X \ar[d]\\
    X' \ar[r] & X}$$
  is a cartesian diagram of properly stable stacks and good moduli spaces, then the image of a primarily strong cycle is primarily strong. Using the methods of Section \ref{sec:op-->strong} this would imply that the image of $A^*_{\op}(X)$ is contained in $A^*_{\pst}(\X/X)_\QQ$.
\end{remark}

\begin{example}
  We note it is not obvious how to generalize Theorem \ref{thm:conj-pst}(\ref{pst::quot-sing}) from the case where $X$ is smooth to the case where it has finite quotient singularities. The reason for this is that even when
the stack $\X$ is smooth and tame,
topologically strong substacks $\Z\subseteq\X$ need not be primarily strong. For example, let $\X=[\AA^2/(\ZZ/2)]$ with $\ZZ/2$ acting on $\AA^2$ by $(x,y)\mapsto(-x,-y)$. Then $X=\spec k[x^2,xy,y^2]$. If $\Z=V(x)$, then $\pi^{-1}(\pi(Z))=V(x^2,xy)$. Notice that $\Z$ is topologically strong since the reduction of $V(x^2,xy)$ is $V(x)$; however, $V(x^2,xy)$ has an embedded point so $\Z$ is not primarily strong.
\end{example}

We end the paper with two questions:

\begin{question}
If $\X$ is a smooth connected properly stable Artin stack with good moduli space $\pi\colon\X\to X$, then are $A^*_{\st}(\X/X)_\QQ$ and $A^*_{\pst}(\X/X)_\QQ$ equal?
\end{question}

In general, if $k > 0$ and $\alpha \in A^k(\X)$ is the image of an
operational class $c \in A^k_{\op}(X)$ then $\alpha$ must, in addition
to being represented by a topologically strong cycle, satisfy $\alpha
\cdot [BG_\x]=0$ for any residual gerbe $BG_\x \subseteq \X$ where $\x$ is closed
point of $\X$. The reason for this is as follows: let $x$ be the image
of $\x$ in $X$ and let $p$ be the restriction of $\pi$ to a morphism
$\x \to x$. Functorial properties of operational Chow groups imply that
$\alpha \cap [BG_\x] = p^* (c \cap [x])$ where $p^*$ is the flat pullback.
However, $c \cap [x] = 0$ as $c \in A^k_{\op}(X)$ with $k > 0$ and $[x]$ is a zero-cycle.

We ask whether this necessary condition is sufficient.

\begin{question}
Let $k>0$ and $\X$ be a smooth connected properly stable Artin stack with good moduli space $\pi\colon\X\to X$. Suppose $\alpha\in A^k_{\tst}(\X/X)_\QQ$ and $\alpha \cdot BG_{\emph\x}=0$ for all residual gerbes $BG_{\emph\x} \subseteq \X$. Then do we have $\alpha=\pi^*c$ for some $c\in A^k_{\op}(X)_\QQ$? Do we have $\alpha\in A^k_{\pst}(\X/X)_\QQ$?
\end{question}

A natural testing ground for the conjectures and questions raised in this section is the theory of toric Artin stacks \cite{GeSa:15a, GeSa:15b}.


\begin{thebibliography}{EHKV01}

\bibitem[AOV]{AOV:08}
Dan Abramovich, Martin Olsson, and Angelo Vistoli, {\em Tame stacks in positive
  characteristic}, Ann. Inst. Fourier (Grenoble) \textbf{58} (2008), no.~4,
  1057--1091.

\bibitem[Alp]{Alp:13}
Jarod Alper, {\em Good moduli spaces for {A}rtin stacks}, Ann. Inst. Fourier
  (Grenoble) \textbf{63} (2013), no.~6, 2349--2402.

\bibitem[AHR]{AHR:15}
Jarod~Alper, Jack~Hall, and David Rydh, {\em A {L}una \'etale slice theorem for
  algebraic stacks}, arXiv:1504.06467 (2015).

\bibitem[Bou]{Bou:87}
Jean-Fran\c{c}ois Boutot, {\em Singularit\'es rationnelles et quotients par les
  groupes r\'eductifs}, Invent. Math. \textbf{88} (1987), no.~1, 65--68.


\bibitem[CH1]{Corti-Hanamura}
 Alessio Corti and Masaki Hanamura, {\em Motivic decomposition and intersection {C}how groups. I}, Duke Math. J. \textbf{103} (2000), no. 3, 459--522.

\bibitem[CH2]{Corti-HanamuraII}
Alessio Corti and Masaki Hanamura, {\em Motivic decomposition and intersection
  {C}how groups. {II}}, Pure Appl. Math. Q. \textbf{3} (2007), no.~1, Special
  Issue: In honor of Robert D. MacPherson. Part 3, 181--203.



\bibitem[Cox]{Cox:95}
David Cox, {\em The homogeneous coordinate ring of a toric variety}, J. Algebraic Geom. \textbf{4} (1995), no. 1, 17--50.


\bibitem[Ed13]{EdidinHandbook} Dan Edidin, {\em Equivariant geometry and the cohomology of the moduli space of curves}, Handbook of Moduli (G. Farkas and I. Morrison editors) (2013), 259--292.

\bibitem[Ed16]{Edi:16}
Dan Edidin, {\em Strong regular embeddings of {D}eligne-{M}umford stacks and
  hypertoric geometry}, Michigan Math. J. \textbf{65} (2016), no.~2, 389--412.

\bibitem[EGS]{EGS:13}
Dan Edidin, Anton Geraschenko, and Matthew Satriano, {\em There is no degree
  map for 0-cycles on {A}rtin stacks}, Transform. Groups \textbf{18} (2013),
  no.~2, 385--389.

\bibitem[EG]{EdGr:98}
Dan Edidin and William Graham, {\em Equivariant intersection theory}, Invent.
  Math. \textbf{131} (1998), no.~3, 595--634.


\bibitem[EM1]{EdMo:12}
Dan Edidin and Yogesh Moore, {\em Partial desingularizations of good moduli
  spaces of {A}rtin toric stacks}, Michigan Math Journal \textbf{61} (2012),
  no.~3, 451--474.

\bibitem[EM2]{EdMo:13}
Dan Edidin and Yogesh More, {\em Integration on {A}rtin toric stacks and
  {E}uler characteristics}, Proc. Amer. Math. Soc. \textbf{141} (2013), no.~11,
  3689--3699.


\bibitem[ER]{EdRy:17}
Dan Edidin and David Rydh, {\em Canonical reduction of stabilizers of {A}rtin
  stacks}, arXiv:1710.03220 (2017).


\bibitem[ES]{EdSa:16} Dan Edidin and Matthew Satriano, {\em Strong cycles and intersection products on good moduli spaces}, in {\em $K$-theory (Mumbai 2016)},
  pp. 225--240, Tata Inst. Fund. Res. Stud. Math {\bf 23}, (2019) arXiv:1609.08162v1.

\bibitem[FR]{Friedlander-Ross}
Eric Friedlander and J. Ross, {\em An approach to intersection theory on singular varieties using motivic complexes}, Compos. Math. \textbf{152} (2016), no. 11, 2371--2404. 

\bibitem[Ful]{Fulton} W. Fulton, {\em Intersection theory}, Second edition. Springer-Verlag, Berlin, 1998

\bibitem[FM]{FulMac:81} William Fulton and Robert MacPherson, {\em Categorical framework for the study of singular spaces}, Mem. Amer. Math. Soc. \textbf{31} (1981), no. 243.

\bibitem[FMSS]{FMSS:95}
W.~Fulton, R.~MacPherson, F.~Sottile, and B.~Sturmfels, {\em Intersection
  theory on spherical varieties}, J. Algebraic Geom. \textbf{4} (1995), no.~1,
  181--193.

\bibitem[FS]{FuSt:97} William Fulton and Bernd Sturmfels, {\em Intersection theory on toric varieties}, Topology \textbf{36} (1997), no. 2, 335--353.

\bibitem[GS1]{GeSa:15a} Anton Geraschenko and Matthew Satriano, {\em Toric stacks I: {T}he theory of stacky fans}, Trans. Amer. Math. Soc. \textbf{367} (2015), no. 2, 1033--1071.

\bibitem[GS2]{GeSa:15b} Anton Geraschenko and Matthew Satriano {\em Toric stacks II: {I}ntrinsic characterization of toric stacks}, Trans. Amer. Math. Soc. \textbf{367} (2015), no. 2, 1073--1094.

\bibitem[Gil]{Gil:84}
Henri Gillet, {\em Intersection theory on algebraic stacks and
  {$Q$}-varieties}, Proceedings of the {L}uminy conference on algebraic
  {$K$}-theory ({L}uminy, 1983), vol.~34, 1984, pp.~193--240.

\bibitem[EGA4]{EGA4}
A.~Grothendieck and J.~Dieudonn\'e, {\em \'{E}lements de {G}\'eom\'etrie
  {A}lg\'ebrique {I}{V}. \'{E}tude locale des schemas et des morphismes de
  sch\'emas}, Inst. Hautes \'Etudes Sci. Publ. Math. No. \textbf{20, 24, 28,
  32} (1964, 1965, 1966, 1967).

\bibitem[Har]{Hartshorne} Robin Hartshorne, {\em Algebraic geometry}, Graduate Texts in Mathematics, No. 52. Springer-Verlag, New York-Heidelberg, 1977.


\bibitem[Kim]{Kim:90} Shun-ichi Kimura, {\em Fractional intersection and bivariant theory}, Comm. Algebra \textbf{20} (1992), no. 1, 285--302.

\bibitem[KM]{KoMo:92}
J\'anos Koll\'ar and Shigefumi Mori, {\em Classification of three-dimensional
  flips}, J. Amer. Math. Soc. \textbf{5} (1992), no.~3, 533--703.

\bibitem[Kre]{Kre:99}
Andrew Kresch, {\em Cycle groups for {A}rtin stacks}, Invent. Math.
  \textbf{138} (1999), no.~3, 495--536.


\bibitem[MFK]{MFK:94}
D.~Mumford, J.~Fogarty, and F.~Kirwan, {\em Geometric invariant theory}, third
  ed., Springer-Verlag, Berlin, 1994.

  
\bibitem[Stacks]{stacks-project} The Stacks Project, \url{http://stacks.math.columbia.edu}, 2016.


\bibitem[Tot]{Totaro}
Burt Totaro, {\em Chow groups, {C}how cohomology, and linear varieties}, Forum
  Math. Sigma \textbf{2} (2014), e17, 25.

\bibitem[Vis]{Vis:89}
Angelo Vistoli, {\em Intersection theory on algebraic stacks and on their
  moduli spaces}, Invent. Math. \textbf{97} (1989), no.~3, 613--670.



\end{thebibliography}
\end{document}